\documentclass[12pt]{amsart}

\usepackage{amssymb,amsthm,amsmath,amstext,amsxtra}
\usepackage{fullpage}
\usepackage{bm}       
\usepackage{mathtools} 
\usepackage{booktabs}
\usepackage{hyperref}
\hypersetup{colorlinks=true,urlcolor=blue,citecolor=blue,linkcolor=blue}
\usepackage{enumerate}
\usepackage{enumitem}
\usepackage{comment}
\usepackage{colonequals}
\usepackage{subcaption}
\usepackage{breqn}
\usepackage{longtable}
\usepackage{numprint}
\usepackage{multirow}

\usepackage{eqparbox}
\usepackage[numbers]{natbib}

\numberwithin{equation}{section}

\theoremstyle{plain}
\newtheorem*{theorema*}{Theorem A}

\newenvironment{introtheorem2}[1]{\noindent\textbf{#1.}\em}{}

\newtheorem{theorem}[equation]{Theorem}

\newtheorem{proposition}[equation]{Proposition}
\newtheorem{lemma}[equation]{Lemma}
\newtheorem{corollary}[equation]{Corollary}

\theoremstyle{definition}
\newtheorem{definition}[equation]{Definition}

\theoremstyle{remark}
\newtheorem{remark}[equation]{Remark}

\newcommand{\Z}{\mathbb{Z}}
\newcommand{\Q}{\mathbb{Q}}

\newcommand{\F}{\mathbb{F}}

\let\P\relax
\newcommand{\P}{\mathbb{P}}

\newcommand{\defi}[1]{\textsf{#1}} 	

\DeclareMathOperator{\Galois}{Gal}
\newcommand{\Gal}[2]{\Galois(#1/#2)}

\DeclareMathOperator{\Frob}{Frob}

\newcommand{\et}{\mathrm{\acute{e}t}}

\DeclareMathOperator{\alg}{{alg}}

\usepackage{xcolor}
\definecolor{darkred}{HTML}{CC1F1F}
\definecolor{green}{rgb}{.4,.7,.4}
\definecolor{blue}{rgb}{.2,.6,.75}
\definecolor{pastelb}{HTML}{3333FF}
\definecolor{pastelyellow}{rgb}{0.992157, 0.552941, 0.235294}
\definecolor{pastelorange}{rgb}{0.941176, 0.231373, 0.12549}
\definecolor{pastelred}{rgb}{0.741176, 0., 0.14902}
\definecolor{darkbrown}{rgb}{0.25098, 0., 0.0745098}

\usepackage{hyperref}
\hypersetup{pdftitle={Restrictions on Weil polynomials of Jacobians of hyperelliptic curves},
pdfauthor={Edgar Costa, Ravi Donepudi, Ravi Fernando, Valentijn Karemaker, Caleb Springer, and Mckenzie West}}
\hypersetup{colorlinks=true,linkcolor=blue,anchorcolor=blue,citecolor=blue}

\begin{document}

\title{Restrictions on Weil polynomials of Jacobians of hyperelliptic curves}

\author{Edgar Costa}
\address{Department of Mathematics, Massachusetts Institute of Technology, Cambridge, MA 02139, USA}
\email{edgarc@mit.edu}
\urladdr{\url{https://edgarcosta.org}}

\author{Ravi Donepudi}
\address{Department of Mathematics, University of Illinois at Urbana-Champaign, Urbana, IL, 61801, USA}
\email{donepud2@illinois.edu}
\urladdr{\url{https://faculty.math.illinois.edu/~donepud2/}}

\author{Ravi Fernando}
\address{Department of Mathematics, University of California --- Berkeley, Berkeley, CA 94720, USA}
\email{fernando@berkeley.edu}
\urladdr{\url{https://math.berkeley.edu/~fernando/}}

\author{Valentijn Karemaker}
\address{Mathematical Institute, Utrecht University, 3508 TA Utrecht, the Netherlands}
\email{V.Z.Karemaker@uu.nl}
\urladdr{\url{http://www.staff.science.uu.nl/~karem001/}}

\author{Caleb Springer}
\address{Department of Mathematics, The Pennsylvania State University, University Park, PA 16802, USA}
\email{cks5320@psu.edu}
\urladdr{\url{http://personal.psu.edu/cks5320/}}

\author{Mckenzie West}
\address{Department of Mathematics, University of Wisconsin-Eau Claire, Eau Claire, WI 54701, USA}
\email{westmr@uwec.edu}
\urladdr{\url{https://people.uwec.edu/westmr/}}

\begin{abstract}
Inspired by experimental data, we investigates which isogeny classes of abelian varieties defined over a finite field of odd characteristic contain the Jacobian of a hyperelliptic curve.  We provide a necessary condition by demonstrating that the Weil polynomial of a hyperelliptic Jacobian must have a particular form modulo 2.  For fixed ${g\geq1}$, the proportion of isogeny classes of $g$-dimensional abelian varieties defined over $\F_q$ which fail this condition is $1 - Q(2g + 2)/2^g$ as $q\to\infty$ ranges over odd prime powers, where $Q(n)$ denotes the number of partitions of $n$ into odd parts.
\end{abstract}

\date{\today}

\maketitle

\section{Introduction}

The question of which abelian varieties arise as Jacobians of curves has a long and rich history.
It has classically been investigated over the complex numbers as the Schottky Problem, using techniques from differential geometry and Hodge theory.
In positive characteristic, some of these tools are no longer available, but instead the Frobenius endomorphism becomes a formidable weapon.

In this article, we study $g$-dimensional abelian varieties which are defined over a finite field $\F_q$ of odd cardinality.
The $\F_q$-isogeny class of such an abelian variety $A$ is uniquely determined by the characteristic polynomial of its Frobenius endomorphism.
This polynomial is called the \defi{Weil polynomial} of the abelian variety and has the following form:
\begin{equation*}
  Z_A(t) =
    t^{2g} + a_1 t^{2g-1} +
    \cdots
    + a_{g-1} t^{g + 1} + a_g t^g + a_{g-1} q t^{g-1} +
    \cdots
    + a_1 q^{g-1}t
    + q^g
    \in \Z[t].
\end{equation*}
Given $Z_A(t)$, one would like to to determine if the isogeny class of $A$ contains the Jacobian of a smooth curve over $\F_q$.

In genus one, the problem is straightforward, as every one-dimensional abelian variety is an elliptic curve, which is isomorphic to its Jacobian.
For genus two, the problem was solved by Howe, Nart, and Ritzenthaler~\cite{hnr} who give an explicit classification using only elementary restrictions involving the integers $a_1, a_2$, and $q$.
Their method relies on three key facts:
\begin{enumerate}
    \item the Jacobian of a curve is canonically principally polarized;
    \item a principally polarized abelian surface is isomorphic to exactly one of the following: a Jacobian of a genus two curve, a product of two elliptic curves,  or the restriction of scalars of an elliptic curve from a quadratic extension of the ground field;
    \item every genus two curve is hyperelliptic.
\end{enumerate}
 In particular, they take advantage of the canonical involution associated with any genus two curve.

Solving this problem in higher genus appears to be significantly more complicated.
For example, every curve of genus three is isomorphic to either a hyperelliptic curve or to a smooth plane quartic curve.
Curves of the latter type generically do not possess any non-trivial automorphisms, so the arguments of Howe, Nart and Ritzenthaler for genus two cannot easily be extended to non-hyperelliptic genus three curves.

In this paper, we focus on a more accessible question, namely, whether a given isogeny class of abelian varieties contains the Jacobian of a hyperelliptic curve.
We do this by studying the geometric configurations of the Weierstrass points of hyperelliptic curves, as is done in the $g=2$ case, e.g., \cite[Appendix]{mn}.
In doing so, we obtain parity conditions on the coefficients of the Weil polynomials which prevent certain isogeny classes from containing the Jacobian of a hyperelliptic curve.
For instance, in genus three we obtain:

\begin{introtheorem2}{Theorem~\ref{main-thm-g3}}
Let $q$ be an odd prime power. The isogeny classes of three-dimensional abelian varieties corresponding to Weil polynomials of the form
\[
t^6 + a_1 t^5 + a_2 t^4 + a_3 t^3 + q a_2 t^2 + q^2 a_1 t + q^3
\]
with  $a_2\equiv 0 \pmod 2$ and $a_3\equiv 1 \pmod 2$ do not contain the Jacobian of a hyperelliptic curve over $\F_q$.
\end{introtheorem2}

In fact, we can perform the same analysis for abelian varieties of any dimension; as far as we know, this is the first non-existence result for Jacobians in isogeny classes when $g > 2$ and $q$ is any odd prime power.
Indeed, we demonstrate that the Weil polynomial of the Jacobian of a hyperelliptic curve of genus $g$ over a finite field of odd characteristic must be congruent modulo 2 to a polynomial of the form $\prod_{i = 1}^r (t^{d_i} - 1)/(t - 1)^2 \in \F_2[t]$ where $2g + 2 = d_1 + \dots + d_r$ is a partition.
We call a Weil polynomial \defi{admissible} if it takes this form modulo 2, and \defi{inadmissible} otherwise.  Notice that it is easy to test whether a polynomial is admissible by explicitly trying all possible partitions, and an isogeny class is guaranteed to not contain a hyperelliptic Jacobian if its Weil polynomial is inadmissible.

By  using asymptotic results on the number of Weil polynomials of a fixed degree whose coefficients lie in prescribed congruence classes modulo an integer, it is possible to determine the asymptotic proportion of isogeny classes with admissible, or inadmissible, Weil polynomials.

\begin{introtheorem2}{Theorem~\ref{thm:asymptotic_hyper}}
Let $c(q,g)$ be the proportion of isogeny classes of $g$-dimensional abelian varieties over $\F_q$ with admissible Weil polynomial.
For all $g \geq 2$ we have
\[
	\lim_{q\to\infty} c(q,g) = \frac{Q(2g +2)}{2^g},
\]
as $q$ ranges over odd prime powers, where $Q(2g+2)$ is the number of partitions of $2g + 2$ into distinct parts.
\end{introtheorem2}

For comparison, notice that for fixed $g > 6$, it is possible to see by simply counting the number of isogeny classes of $g$-dimensional abelian varieties over $\F_q$ versus the number of hyperelliptic curves of genus $g$ over $\F_q$  that actually $0\%$ of $g$-dimensional abelian varieties over $\F_q$ contain a hyperelliptic Jacobian as $q\to\infty$ ranges over odd prime powers.
However, this counting argument does not provide a way to identify which isogeny classes do or do not contain a hyperelliptic Jacobian.
In contrast, for any fixed $g \geq 1$, any isogeny class with an inadmissible Weil polynomial is explicitly known to not contain a hyperelliptic Jacobian.
Rephrasing the theorem above for fixed $g > 6$, the proportion of isogeny classes of $g$-dimensional abelian varieties over $\F_q$ which do not contain a hyperelliptic Jacobian but are not found by testing for inadmissible Weil polynomials is $Q(2g + 2)/2^g$ as $q\to\infty$ ranges over odd prime powers. As noted in Remark \ref{rem:disc}, this discrepancy approaches zero as $g$ grows.

On the other hand, over the algebraic closure $\overline{\F}_q$, determining which isogeny classes contain Jacobians remains an open question. 
In order for our methods to shed light in this context, the behavior of admissibility under finite field extensions must first be understood. We leave this question for future work.

The outline of the paper is as follows.
In Section~\ref{sec:weilpolymod2}, we recall some results on the geometry of hyperelliptic curves and their Jacobians and prove
Theorem~\ref{main-thm-g3} and its generalizations to higher genus.
In Section~\ref{sec:asymptotics}, we study asymptotic consequences of these non-existence results, deriving
Theorem~\ref{thm:asymptotic_hyper}.
In Section~\ref{sec:pointcounts}, we determine restrictions on point counts of curves, allowing us to reprove Theorem~\ref{main-thm-g3} by elementary methods.
Finally, in Section~\ref{sec:experimentaldata}, we present experimental data on the optimality of our results.

\subsection*{Acknowledgements}
This project started at the AMS MRC workshop ``Explicit methods in arithmetic geometry in characteristic $p$'' held at Whispering Pines (RI).
The authors thank Andrew Sutherland for suggesting this problem, and Jeff Achter, Everett Howe, Bjorn Poonen, Christophe Ritzenthaler, David Roe, Andrew Sutherland, and Christelle Vincent for helpful discussions.
They also thank the anonymous referees for careful reading and useful feedback.
The first author was supported by the Simons Collaboration in Arithmetic Geometry, Number Theory, and Computation via Simons Foundation grant 550033. 
The third author was partially supported by NSF RTG grant DMS-1646385.
The fourth author was supported by an AMS--Simons Travel Grant.
The fifth author was partially supported by National Science Foundation award CNS-1617802.

\section{Weil polynomials mod 2}\label{sec:weilpolymod2}

Let $C$ be a hyperelliptic curve of genus $g$ defined over $\F_q$ (of characteristic $p \neq 2$) with canonical degree two map $\pi\colon C\to \P^1$.
Let $W \colonequals \{\alpha_1,\alpha_2,\dots ,\alpha_{2g+2}\}$ be the support of ramification divisor of $\pi$, i.e., the geometric Weierstrass points of $C$.
The $\F_q$-Frobenius endomorphism $\Frob$ acts on $W$ by permuting its elements; we denote the cardinalities of the orbits of this action by $d_i$.
(Alternatively, $W$ consists of some number, $r$, of $\F_q$-places with respective degrees $d_i$, where $d_1 + \cdots +d_r = 2g+2$.)

The multiset $d_C \coloneqq \{d_i\}$ forms a partition of the integer $2g+2$; we call this the \defi{degree set} of the curve $C$, and by convention we order the $d_i$ so that $d_1 \leq d_2 \leq \cdots \leq d_r$.

Explicitly, $C$ can be given in coordinates by $y^2 = f(x)$, where $f$ is a squarefree polynomial of degree either $2g+1$ or $2g+2$.
Then $W$ consists of the points $(\alpha, 0)$, where $\alpha$ runs over the roots of $f$, together with the point at infinity if $\deg(f) = 2g+1$.
Then the $d_i$ are precisely the degrees of the irreducible factors of $f$, along with an extra 1 in the case $\deg(f) = 2g+1$.

\begin{remark}\label{rem:smallqpartititions}
If $q$ is small compared to the genus, not all partitions of $2g+2$ may arise as degree sets of genus $g$ hyperelliptic curves, due to the finite number of irreducible polynomials of any fixed degree in $\F_q[x]$.
For example, if $2g + 2 > q + 1$, then the partition $2g + 2 = 1 + \cdots + 1$ cannot be realized, as there are not enough $\F_q$-points for $C$ to ramify over.
\end{remark}

Let $J$ denote the Jacobian variety of $C$ whose elements are degree-zero divisors on $C$ modulo linear equivalence.
Then the group $J[2]$ of (geometric) $2$-torsion of $J$ is an $\F_2$-vector space that admits the following explicit description.

\begin{lemma}\label{lem:quotient}
  The group $J[2]$ of $2$-torsion elements of $J$ forms a $2g$-dimensional vector space over $\F_2$.
  Explicitly, this group can be expressed as the vector space obtained from $\F_2^{2g+2}$ by considering all vectors with an even number of non-zero entries and forming the quotient by $\langle(1,1,\ldots, 1)\rangle$.
\end{lemma}
\begin{proof}
The first claim follows since $p\neq 2$.
For the second, we argue as follows; cf.\ \cite[Corollary 2.11]{mumford} over $\mathbb{C}$, which extends to any algebraically closed field of characteristic unequal to two.
Every element of $J[2]$ can be explicitly represented either by a divisor
\[
 e_U\coloneqq\sum_{P\in U}P -|U|(\infty),
 \]
if the map $\pi\colon C\to \P^1$ is ramified at infinity, in which case the symbol $\infty$ slightly abusively denotes $\pi^{-1}(\infty) \in C$, or by a divisor
  \[
 e_U\coloneqq\sum_{P\in U}P -\frac{|U|}{2}(\infty_1+ \infty_2),
\]
if $\pi$ is split at infinity, in which case $\pi^{-1}(\infty) = \{ \infty_1, \infty_2\}$.
Either way, the set $U\subseteq W$ is a subset of even cardinality, and two divisors, $e_U$ and $e_{U'}$, represent the same element of $J[2]$ if either $U=U'$ or $U=W\setminus U'$.
In the $\F_2$-vector space $V$ whose standard basis is indexed by $W$, the above set of representatives determine the subspace of vectors with an even number of non-zero entries.
Two vectors $v_1,v_2\in V$ yield the same element of $J[2]$ precisely when their sum is contained in $\langle(1,1,\dots, 1) \rangle$.
The second claim follows.
\end{proof}

We will also need the following standard fact about vector spaces, whose proof is omitted.

\begin{lemma}\label{lem:ses}
Consider an exact sequence of vector spaces
\[
0\to W_1 \to W_2 \to W_3 \to 0
\]
and a linear map $T\colon W_2\to W_2$ such that $T(W_1)\subseteq W_1$.
We will also denote the induced map $T\colon W_3\to W_3$.
For $i \in \{1,2,3\}$, denote the characteristic polynomial of $T$ on $W_i$ by $\chi(T,W_i)$.
Then we have
\[
\chi(T,W_2)=\chi(T,W_1)\chi(T,W_3).
\]
\end{lemma}

Denote the base-change of $C$ (resp.\ $J$) to $\overline{\F}_q$ by $C^{\alg}$ (resp.\ $J^{\alg}$).
Exploiting the action of $\Frob$ on the geometric Weierstrass points $W$, we obtain the following result on the characteristic polynomial of $\Frob$ modulo $2$.

\begin{proposition}\label{prop:2torsion}
  Let $C$ be a hyperelliptic curve of genus $g$ defined over $\F_q$.
  Let $\{d_i\}_{i=1}^r$ be the partition of $2 g + 2$ which records the sizes of the orbits of Frobenius acting on the $2 g + 2$ geometric Weierstrass points.
  For any prime $\ell \neq p$ we have
  \begin{equation}
  \label{eq:2torsion}
    \det\left( 1 - \Frob t \,|\, H_{\et}^1( C^{\alg}, \Q_{\ell}) \right)
    \equiv
    \Big(\prod_{i=1}^r {t}^{d_i} - 1 \Big) / {(t-1)}^2 \pmod{2}.
  \end{equation}
\end{proposition}

\begin{proof}
We have (cf.\ \cite[Corollary 9.6]{milne}) that
  \begin{align*}
    \det\left( 1 - \Frob t \,|\, H_{\et}^1( C^{\alg}, \Q_{\ell}) \right)
    &=
    \det\left( 1 - \Frob t \,|\, H_{\et}^1( J^{\alg}, \Q_{\ell}) \right)\\
    &=
    \det\left( 1 - \Frob t \,|\, V_\ell(J) \right),
  \end{align*}
where $V_{\ell}(J) = T_{\ell}(J) \otimes_{\Z_{\ell}} \Q_{\ell}$, and where $T_{\ell}(J)$ is the $\ell$-adic Tate module of $J$.
By the Weil conjectures the polynomials above all have integer coefficients.
By taking $\ell = 2$, we further have
\[
    \det\left( 1 - \Frob t \,|\, V_2(J) \right)
    \equiv
    \det\left( 1 - \Frob t \,|\, J[2] \right) \pmod{2}.
\]

As before, consider the $\F_2$-vector space $W_2\colonequals \F_2 ^{2g+2} $ whose standard basis is indexed by $W$ and which is acted on by the $\F_q$-Frobenius endomorphism $\Frob$.
By assumption, the action of $\Frob$ on $W_2$ can be represented by a block-diagonal matrix with $r$ blocks, whose $i$th block (of order $d_i$) is a cyclic permutation of basis vectors.
Since the characteristic polynomial of a cyclic permutation of order $n$ is $t^n-1$, it follows that
\begin{equation*}
\chi(\Frob, W_2)=\prod_{i=1}^r \bigl(t^{d_i}-1\bigr).
\end{equation*}
Furthermore, let $W_1$ be the codimension-one subspace of vectors with an even number of non-zero entries.
Then $W_1$ is stable under $\Frob$.
Moreover, $\Frob$ acts invertibly on $W_1$ and hence also on the one-dimensional quotient $W_3\colonequals W_2/W_1$, so that $\chi(\Frob,W_3) = t-1$.
Applying Lemma~\ref{lem:ses} to the short exact sequence $0 \to W_1 \to W_2 \to W_2/W_1 \to 0$ yields
\[
\chi(\Frob,W_1)=\frac{\prod_{i=1}^r (t^{d_i}-1)}{t-1}.
\]
Now let $W_4$ be the one-dimensional subspace of $\langle (1, 1, \ldots, 1)\rangle \subseteq W_1$ and define $W_5 \coloneqq W_1/W_4$.
Similarly, $W_4$ is stable under $\Frob$ with $\chi(\Frob, W_4) = t-1$ since $\Frob$ acts invertibly.
Applying Lemma~\ref{lem:ses} to the short exact sequence $0\to W_4\to W_1\to W_5 \to 0$, we find
\[
\chi(\Frob, W_5)=\frac{ \prod_{i=1}^r \bigr(t^{d_i}-1\bigl)}{(t-1)^2}.
\]
Since, by Lemma~\ref{lem:quotient}, $\Frob$ acts on $W_5$ as it does on $J[2]$, the result follows.
\end{proof}

\begin{remark}
It is possible to prove a congruence modulo 2 similar to \eqref{eq:2torsion} in Proposition~\ref{prop:2torsion} also when the finite field $\mathbb{F}_q$ is allowed to have characteristic 2. In this case, Frobenius acts on Weierstrass points as it does on the {\'e}tale part of $J[2]$. Since we require odd characteristic throughout the rest of the paper, notably in Theorem~\ref{thm:asymptotic}, we do not pursue this further.
\end{remark}

\begin{definition}
We call any polynomial of the form~\eqref{eq:2torsion} an \defi{admissible Weil polynomial modulo 2}.
Note that the notion of admissibility is independent of $q$.
\end{definition}

By applying Proposition~\ref{prop:2torsion} to all possible partitions of $8$ we obtain our main theorem.
\begin{theorem}\label{main-thm-g3}
Let $q$ be an odd prime power.
The isogeny classes of three-dimensional abelian varieties corresponding to Weil polynomials of the form
\[
t^6 + a_1 t^5 + a_2 t^4 + a_3 t^3 + q a_2 t^2 + q^2 a_1 t + q^3
\]
with  $a_2\equiv 0 \pmod 2$ and $a_3\equiv 1 \pmod 2$ do not contain the Jacobian of a hyperelliptic curve over $\F_q$.
\end{theorem}

\begin{proof}
  For each of the twenty-two partitions $\{d_i\}$ of $2g+2 = 8$, we compute the corresponding polynomial $(t-1)^{-2} \prod_{i=1}^r (t^{d_i}-1)$.
By Proposition~\ref{prop:2torsion}, these are all of the admissible Weil polynomials modulo $2$ for Jacobians of hyperelliptic curves of genus three.
In Table~\ref{tab:partitions_genus_3}, we tabulate the coefficients $(a_1,a_2,a_3)$ of the resulting Weil polynomials modulo $2$ with the corresponding partitions $\{d_i\}$.
\begin{table}[!ht]
\addtocounter{equation}{1}
\begin{tabular}{ c | l }
  Coefficients $(a_1,a_2,a_3) \pmod{2}$ & Partition of 8\\
\hline
$(0, 1, 1)$ & $\{3, 5\}$ \\
$(1, 1, 0)$ & $\{1, 1, 1, 1, 1, 3\}$, $\{1, 1, 1, 2, 3\}$, $\{1, 2, 2, 3\}$, $\{1, 3, 4\}$ \\
$(1, 0, 0)$ & $\{1, 1, 1, 5\}$, $\{1, 2, 5\}$ \\
$(0, 0, 0)$ & $\{1, 1, 3, 3\}$, $\{1, 1, 6\}$, $\{2, 3, 3\}$, $\{2, 6\}$ \\
$(0, 1, 0)$ & $\{1, 1, 1, 1, 1, 1, 1, 1\}$, $\{1, 1, 1, 1, 1, 1, 2\}$, $\{1, 1, 1, 1, 2, 2\}$, \\
& $\{1, 1, 1, 1, 4\}$, $\{1, 1, 2, 2, 2\}$, $\{1, 1, 2, 4\}$,\\
&  $\{2, 2, 2, 2\}$, $\{2, 2, 4\}$, $\{4, 4\}$, $\{8\}$ \\
$(1, 1, 1)$ & $\{1, 7\}$
\end{tabular}
\caption{Weil coefficients modulo 2 and corresponding partitions for threefolds.}
\label{tab:partitions_genus_3}
\end{table}
\end{proof}

Although Theorem~\ref{main-thm-g3} only considers the case of three-dimensional abelian varieties, Proposition~\ref{prop:2torsion} applies much more generally.
Indeed, for any $g \geq 1$, we can produce a list of admissible Weil polynomials modulo 2 for Jacobians of hyperelliptic curves of genus~$g$, independent of $q$.
The following result counts the number of admissible polynomials in terms of partitions.

\begin{proposition}\label{prop:countingpolys}
The number of admissible Weil polynomials modulo 2 for Jacobians of hyperelliptic curves of genus $g$ over finite fields of odd characteristic is equal to $Q(2g + 2)$, the number of partitions  of $2g + 2$ into distinct parts, or equivalently, the number of partitions of $2g+2$ into odd parts.
 \end{proposition}

\begin{proof}
It is well-known that the number of partitions of $2g+2$ into distinct parts is equal to the number of partitions of $2g + 2$ into odd parts; see \cite[Theorem 10.2]{nzm}.  For the remainder of this proof, our perspective will focus on partitions into distinct parts.

 Proposition~\ref{prop:2torsion} shows how to compute admissible Weil polynomials modulo 2 using partitions, although many different partitions can correspond to a single admissible polynomial.
 For the purposes of this proof, we will call two partitions $\{d_1, \dots, d_r\}$ and $\{e_1,\dots, e_s\}$ of an integer $n$ \defi{equivalent} if
 \begin{equation*}
 	\prod_{i = 1}^r \bigl(t^{d_i} - 1\bigr) \equiv \prod_{j = 1}^s \bigl(t^{e_j} - 1\bigr) \pmod 2.
\end{equation*}
 Using Equation~\eqref{eq:2torsion}, it suffices to prove that every equivalence class of partitions of $2g + 2$ contains precisely one partition with distinct parts.

Observe that every partition of $2g + 2$ is equivalent to a partition with distinct parts.
Indeed, if $\{d_1, \dots ,d_{r+1}\}$ is any partition of $2g + 2$ where $d_{r} = d_{r + 1}$, then we can construct another partition  $\{d_1, \dots , d_{r-1}, d'_r\}$ with $d'_{r} = 2d_r$.
These two partitions are equivalent since $t^{2d_r} -1 \equiv (t^{d_r} - 1)^2 \mod 2$.
Thus by induction on the number of equal parts, we conclude that every admissible Weil polynomial modulo 2 arises as a partition of $2g + 2$ into distinct parts.

Now suppose that $\{d_1, \dots, d_r\}$ and $\{e_1, \dots, e_s\}$ are two equivalent partitions of $2g +2$ into distinct parts.
Without loss of generality, we order the parts so that $d_1$ and $e_1$ are the smallest parts of their respective partitions.
After expanding the polynomials above, we see that the non-constant monomials of smallest degree are $t^{d_1}$ and $t^{e_1}$, respectively.
This implies that $d_1 = e_1$, hence $\{d_2, \dots, d_r\}$ and $\{e_2, \dots, e_s\}$ are equivalent partitions of $2g +2 - d_1$.
Thus by induction on $r$, we conclude that two equivalent partitions of $2g + 2$ into distinct parts are equal.
\end{proof}

\begin{remark}
  In contrast with Remark~\ref{rem:smallqpartititions}, every admissible Weil polynomial mod 2 arises as the reduction mod 2 of a Weil polynomial for a hyperelliptic curve over $\F_q$, for each $q$:
Since there is at least one irreducible polynomial of every degree, we
  can construct a hyperelliptic curve ramified at points of degree $d_i$, where $\{d_i\}$ is the partition of $2g +2$ into distinct parts.
\end{remark}

In Table~\ref{tab:partition_table}, we tabulate the number of admissible and inadmissible Weil polynomials modulo $2$ for Jacobians of hyperelliptic curves of genus $g\leq 7$, i.e., the numbers $Q(2g +2)$ and $2^g - Q(2g +2)$.

\begin{table}[!ht]
\addtocounter{equation}{1}
\begin{tabular}{r|ccccccc}
$g$ & 1 & 2 & 3 & 4 & 5 & 6 & 7 \\
\hline
\multirow{2}{*}{$Q(2g + 2)$} & 2 & 4 & 6 & 10 & 15 & 22 & 32 \\
 & 100.00\% & 100.00\% & 75.00\% & 62.50\% & 46.88\% & 34.38\% & 25.00\% \\
\hline
\multirow{2}{*}{$2^g - Q(2g + 2)$} & 0 & 0 & 2 & 6 & 17 & 42 & 96 \\
 & 0.00\% & 0.00\% & 25.00\% & 37.50\% & 53.12\% & 65.62\% & 75.00\% \\
\end{tabular}
%
%

\caption{The number of admissible  and inadmissible Weil polynomials modulo $2$ for hyperelliptic curves of small genus.}
\label{tab:partition_table}
\end{table}

\section{Asymptotics}\label{sec:asymptotics}

The results in the previous section show that an isogeny class cannot contain a hyperelliptic Jacobian if the coefficients of the corresponding Weil polynomial satisfy a certain parity condition. Therefore, it is natural to ask how many Weil polynomials satisfy any given parity condition, or more generally, how many Weil polynomials are congruent to a given fixed polynomial of the correct form modulo an integer $m$.  The following theorem answers this question asymptotically for abelian varieties of fixed dimension.

\begin{theorem}\label{thm:asymptotic}
  Let $g\geq 2$ and $m$ be fixed integers, and let $f(t) \in \Z[t]$ be a fixed polynomial of the form
  \begin{equation*}
    f(t) = t^{2g} + a_1 t^{2g-1} +
    \cdots
    + a_{g-1} t^{g + 1} + a_g t^g + a_{g-1} q t^{g-1} +
    \cdots
    + a_1 q^{g-1}t
    + q^g.
  \end{equation*}
  For a prime power $q$ coprime to $m$, write $e_{f,m}(q,g)$ for the proportion of isogeny classes of $g$-dimensional abelian varieties over $\F_q$ whose Weil polynomial is congruent to $f(x)$ modulo~$m$.
  Then
  \begin{equation*}
    \lim_{q\to \infty} e_{f,m}(q,g) = \frac{1}{m^g},
  \end{equation*}
  where the limit is taken over all prime powers $q$ coprime to $m$.
\end{theorem}

\begin{proof}
  This is essentially a theorem of Holden~\cite[Theorem~5]{holden}, although we present a slightly more general statement here.
  Specifically, Holden's theorem only considers the case when $m =\ell$ is a prime, and the limit is taken over $q = p^r$ for a single fixed prime $p$.
  However, the more general version of the theorem presented above is obtained immediately from Holden's methods, as follows.
  Denote by $I(q,g)$ the number of isogeny classes of $g$-dimensional abelian varieties over $\F_q$ and write $I_{f,m}(q,g)$ for the number of such isogeny classes whose Weil polynomial is congruent to $f(x)$ modulo $m$.
  With this notation we may write $e_{f,m}(q,g) = \frac{I_{f,m}(q,g)}{I(q,g)}$.
  Using lattices, DiPippo and Howe obtained upper and lower bounds for $I(q,g)$; see~\cite[Theorem 1.2]{dh} and~\cite{dh-correction}.
  Holden employed analogous techniques to obtain similar bounds for $I_{f,m}(q,g)$; see~\cite[Proposition 2.2]{holden}.
  Using these bounds, as in the proof of~\cite[Theorem~5]{holden}, we find
  \begin{equation*}
    \begin{aligned}
&
\frac{v_gr(q) q^{g(g+1)/4} m^{-g} - 2c(g,m) q^{g(g+1)/4 - 1/2} m^{1-g}}
{v_gr(q) q^{g(g+1)/4} +(v_g + 3c(g,1)) q^{g(g+1)/4 - 1/2}}\\
&
\hspace{2cm}
\leq \frac{I_{f,m}(q,g)}{I(q,g)}\\
&\hspace{2cm}
\leq \frac{v_gr(q) q^{g(g+1)/4} m^{-g} +(v_g + 3c(g,1)) q^{g(g+1)/4 - 1/2} m^{1-g}}
{v_gr(q) q^{g(g+1)/4}  - 2c(g,m)q^{g(g+1)/4 - 1/2}},
    \end{aligned}
  \end{equation*}
  where $v_g$ is a constant depending $g$, and $c(g,m)$ is a constant depending on both $g$ and $m$, and $r(q) = \varphi(q)/q$, where $\varphi$ denotes Euler totient function.
  Letting $q\to \infty$, the theorem follows.
\end{proof}

Recall that, in light of Theorem~\ref{main-thm-g3} and Proposition~\ref{prop:countingpolys}, we say that the Weil polynomial of an isogeny class of $g$-dimensional abelian varieties is admissible if it is congruent modulo 2 to a polynomial of the form $\prod_{i = 1}^r (t^{d_i} - 1)/(t^2 - 1)\in \F_2[t]$ for some partition $2g + 2 = d_1 + \dots + d_r$, and is inadmissible otherwise. By combining Theorems~\ref{main-thm-g3} and~\ref{thm:asymptotic}, we obtain the following corollary.

\begin{corollary}\label{cor:hq}
Denote by $c(q)$ the proportion of isogeny classes of abelian threefolds over $\F_q$ with admissible Weil polynomials.
Then
\begin{equation*}
	\lim_{q\to\infty} c(q) =75\%,
\end{equation*}
where the limit is taken over all odd prime powers $q$.
 \end{corollary}

 More generally, by combining Theorem~\ref{main-thm-g3} and Proposition~\ref{prop:countingpolys}, we can determine the proportion of isogeny classes of $g$-dimensional abelian varieties over $\F_q$ with admissible Weil polynomials in terms of the number of partitions of $2g + 2$ into distinct parts.

\begin{theorem}\label{thm:asymptotic_hyper}
Let $c(q,g)$ be the proportion of isogeny classes of $g$-dimensional abelian varieties over $\F_q$ with admissible Weil polynomial.
For all $g \geq 2$ we have
\[
	\lim_{q\to\infty} c(q,g) = \frac{Q(2g +2)}{2^g},
\]
as $q$ ranges over odd prime powers, where $Q(2g+2)$ is the number of partitions of $2g + 2$ into distinct parts.
\end{theorem}
 \begin{proof}
By Proposition~\ref{prop:countingpolys}, we know there is a set of polynomials $S = \{f_1, \dots, f_{Q(2g+2)}\}$ such that the Weil polynomial of the Jacobian of a hyperelliptic curve of genus $g$ over any finite field of odd characteristic is equivalent modulo 2 to some polynomial in $S$. Thus, the result follows immediately from Theorem~\ref{thm:asymptotic}.
 \end{proof}

 \begin{remark}
 \label{rem:disc}
For fixed $g\geq 2$, the number of isogeny classes of $g$-dimensional abelian varieties over $\F_q$ and the number of hyperelliptic curves of genus $g$ over $\F_q$ are asymptotically bounded by $q^{g(g+1)/4}$ and $q^{2g-1}$ , respectively, as $q\to\infty$ ranges over odd prime powers; see  \cite[Theorem 1.1]{dh}, and \cite[Table 1, Corollary 3.4]{nart}.
For $g > 6$, this means that asymptotically $0\%$ of the isogeny classes of $g$-dimensional abelian varieties over $\F_q$ contain a hyperelliptic Jacobian as $q\to\infty$ ranges over odd prime powers.
Comparing this reality to Theorem~\ref{thm:asymptotic_hyper}, we see that the proportion of isogeny classes of $g$-dimensional abelian varieties that have an admissible Weil polynomial but still do not contain a hyperelliptic Jacobian is $\frac{Q(2g +2)}{2^g}$ as ${q\to\infty}$.
Notice that this discrepancy gets smaller as $g$ grows. Indeed, $Q(N) \sim \frac{3^{3/4}}{12N^{3/4}}\exp(\pi\sqrt{N/3})$ as $N \to \infty$ \cite[Figure 1.9]{fs}, so it follows that $\frac{Q(2g +2)}{2^g}\to 0$ as $g \to \infty$.
 \end{remark}

 \section{Point counts}\label{sec:pointcounts}

In Section~\ref{sec:weilpolymod2} we determined restrictions modulo $2$ for Weil polynomials of hyperelliptic Jacobians.
These Weil polynomials govern the point counts of the corresponding hyperelliptic curve $C/\F_{q}$ over all extensions of $\F_q$.
In this section, we determine $2$-adic restrictions on the point counts of hyperelliptic curves over extensions of $\F_q$ by more elementary means. Rather than studying the action of Frobenius on $J[2]$, we study its action on the Weierstrass points of the curve directly.
In particular, this provides an alternative proof of Theorem~\ref{main-thm-g3}.

\subsection{Restrictions on parity of point counts}

As before, let $q$ be an odd prime power, $C/\F_q$ be a hyperelliptic curve of genus $g > 1$.
Recall that we denote the support of the ramification divisor by $W$ and we denote by $\{d_i\}$ a partition of $2g + 2$, corresponding to the decomposition of $W$ into Frobenius orbits.
We begin with the following observations on the point counts of $C$ over extensions of $\F_q$.

\begin{lemma}\label{lem:W_count}
For each $n \geq 1$, we have $\#W(\F_{q^n}) = \sum_{i:  d_i | n} d_i$.
\end{lemma}
\begin{proof}
A point of degree $d$ contributes $d$ $\F_{q^n}$-points if $d|n$, and none otherwise.
\end{proof}

\begin{lemma}\label{lem:orbits}
  For each $n \geq 1$, we have $\#C(\F_{q^n}) \equiv \#W(\F_{q^n}) \pmod{2}$.
\end{lemma}
\begin{proof}
The $\F_{q^n}$-points of $W$ are precisely the $\F_{q^n}$-points of $C$ that are fixed by the hyperelliptic involution.  Since all other points appear in pairs, it follows that $\#C(\F_{q^n}) \equiv \#W(\F_{q^n}) \pmod 2$.
\end{proof}

\begin{corollary}
 For each $n \geq 1$, we have $\#C(\F_{q^n}) \equiv \#C(\F_{q^{2n}}) \pmod{2}$.
\end{corollary}
\begin{proof}
  The numbers $n$ and $2n$ share the same set of odd divisors, therefore by Lemmas~\ref{lem:W_count} and~\ref{lem:orbits}
  \begin{equation*}\#C(\F_{q^n}) \equiv \#W(\F_{q^n}) = \sum_{i:  d_i | n} d_i \equiv \sum_{i:  d_i | 2n} d_i  = \#W(\F_{q^{2n}}) \equiv \#C(\F_{q^{2n}}) \pmod{2}.\qedhere
  \end{equation*}
\end{proof}

\begin{corollary}\label{lem:mobius}
For each $n$, we have
\begin{align*}
  n \cdot \#\{i:  d_i = n\} \equiv \sum_{d|n} \mu(n/d) \#C(\F_{q^d}) \pmod 2,
\end{align*}
where $\mu$ is the M\"obius function.
\end{corollary}
\begin{proof}
  Using Lemma~\ref{lem:W_count} we can restate Lemma~\ref{lem:orbits} as
$$ \#C(\F_{q^n}) \equiv \sum_{d | n} d \cdot \#\{i:  d_i = d\} \pmod{2}. $$
The result follows by applying M\"obius inversion.
\end{proof}

\begin{remark}
For $C$ as above, consider the binary sequence
\[
a_C(n)\coloneqq\#C(\F_{q^n}) \pmod{2}.
\]

Lemma~\ref{lem:orbits} and Corollary~\ref{lem:mobius} imply that $a_C$ determines, and is determined by, the numbers $n \cdot \#\{i:  d_i = n\} \pmod{2}$ for all $n$.  These numbers are encoded in the degree set $d_C$;  they tell us the parity of $\#(i:  d_i = n)$ when $n$ is odd, and give no information when $n$ is even.

More precisely, Lemma~\ref{lem:orbits} and Corollary~\ref{lem:mobius} give us a dictionary between the sequence $a_C$ of parities of point counts and the set
\[
\{d:  d \text{ is odd and appears an odd number of times in } d_C\}.
\]
This is consistent with Proposition~\ref{prop:countingpolys}.
\end{remark}

This is already enough to prove that some sequences of point count parities are inconsistent with the identity $\sum d_i = 2g+2$.  For example:

\begin{lemma}\label{lem:not1}
  If the genus of $C$ is 3, then we cannot have $\#C(\F_{q}) \equiv 0, \#C(\F_{q^3}) \equiv~1$, and $\#C(\F_{q^5}) \equiv 0 \pmod 2$.
\end{lemma}

\begin{proof}
  If such a curve $C$ existed, Corollary~\ref{lem:mobius} implies that $\#\{d_i = 1\}$ is even, $\#\{d_i = 3\}$ is odd, and $\#\{d_i = 5\}$ is even.
  Thus $\#\{d_i = 3\} = 1$ and $\#\{d_i = 5\} = 0$, which contradicts $\#\{d_i = 1\}$  being even.
\end{proof}

\subsection{Restrictions on point counts modulo powers of 2}

We have just seen some restrictions on the point counts modulo $2$ of hyperelliptic curves.
In this section we will obtain further restrictions modulo higher powers of 2.

Fix an integer $m \geq 1$.
Let $G$ be the group $\{\pm 1\} \times \Gal{\F_{q^{2^m}}}{\F_q}$ of order $2^{m+1}$.
This acts on the $\F_{q^{2^m}}$-points of $C$, where the first factor acts by the hyperelliptic involution and the second by field automorphisms.  We will study $\#C(\F_{q^{2^m}}) \pmod{2^{m+1}}$ by examining the orbits of the action of $G$.

\begin{proposition}\label{prop:2adicnew}
  If $C$ is a genus $g$ hyperelliptic curve over $\F_q$, then
  \[
  \#C(\F_{q^n}) \equiv 2(q^n + 1) - \#W(\F_{q^n}) \pmod{2^{a+1}},
\]
where $n = 2^a \cdot m \geq 1$ with $m$ odd.
\end{proposition}

\begin{proof}
We will count the $\F_{q^n}$-points of $C$ by considering the fibers of the hyperelliptic map $\pi\colon  C(\F_{q^n}) \to \P^1(\F_{q^n})$.
If $x \in \P^1(\F_{q^n})$, then $\pi^{-1}(x)$ can be computed by extracting the square roots of $f(x)$, where $f$ is a polynomial of degree $2g+1$ or $2g+2$.
Accordingly, $\pi^{-1}(x)$ contains either zero, one, or two $\F_{q^n}$-points, and a preimage of size one occurs if and only if $x \in W$.
Further, if $n$ is even and $x$ is defined over $\F_{q^{n/2}}$ (equivalently, if its degree is not divisible by $2^a$), then the preimage must have size two, since quadratic equations $y^2 = f(x)$ over $\F_{q^{n/2}}$ can be solved over $\F_{q^n}$.
Thus, the preimage may have size zero only if the degree of $x$ is divisible by $2^a$.  In this case, the Galois orbit of $x$ has size divisible by $2^a$, so if there are $\F_{q^n}$-points above $x$, they occur in an orbit whose size is divisible by $2^{a+1}$.
Hence, to count modulo $2^{a+1}$, we may assume that all unramified $\F_{q^n}$-points of $\P^1$ have exactly two preimages defined over $\F_{q^n}$, and subtract the number of $\F_{q^n}$-points of $W$ to obtain the result.
\end{proof}

As with Proposition~\ref{prop:2torsion}, this implies that certain sequences of point counts of hyperelliptic curves (equivalently, certain Weil polynomials of their Jacobians) cannot be realized.

\begin{corollary}\label{cor:counts}
  Assume that the genus of $C$ is $3$.
  If $\#C(\F_q) \equiv 1 \pmod 2$ and $\#C(\F_{q^3}) \equiv 0 \pmod 2$, 
  then we have $\#C(\F_{q^{2^m}}) \equiv -1 \pmod{2^{m+1}}$ for all $m \geq 1$.
\end{corollary}
\begin{proof}
  Since $\#C(\F_q)$ is odd, we have that $\#\{d_i = 1\}$ is odd by Lemma~\ref{lem:orbits}.
  Analogously, since $\#\{d_i = 1\}$ is odd and $\#C(\F_{q^3})$ is even, it follows that $\#\{d_i = 3\}$ is odd and hence must equal one.
  Altogether, we have $d_j = 3$ for a unique $j$ and $d_{j'} \in \{1,2,4\}$ for $j' \neq j$.
  Hence $\# \left( \bigcup_{k \geq 1} W(\F_{q^{2^k}}) \right) = 5$, and so $\# W(\F_{q^{2^m}}) \equiv 5 \pmod{2^{m+1}}$.  The result follows by applying Proposition~\ref{prop:2adicnew} to $n = 2^m$, since we have $q^{2^m} \equiv 1 \pmod{2^{m+1}}$ for $q$ odd.
\end{proof}

\subsection{Reproving Theorem~\ref{main-thm-g3}}

The results in the previous subsections can be used to reprove Theorem~\ref{main-thm-g3} as follows:

 \begin{proof}[Alternative proof of Theorem~\ref{main-thm-g3}]
    Suppose that the statement is false, so that there is some hyperelliptic curve $C$ of genus three whose Jacobian has the prescribed Weil polynomial.
    By using Newton's formulae and \cite[Theorem 11.1]{milne}, we can write the point counts $\#C(\F_{q^k})$ in terms of the coefficients of the Weil polynomial for all $k\geq 1$,  as follows:
\begin{align*}
	\#C(\F_q) &= q +1 + a_1;\\
	\#C(\F_{q^2}) &= q^2+1 -  a_1^2 +2a_2;\\
	\#C(\F_{q^3}) &= q^3+1 +a_1^3 -3a_1a_2 + 3a_3;\\
	\#C(\F_{q^4}) &= q^4+1- a_1^4 + 4a_1^2a_2 - 4a_1a_3 - 2a_2^2 +4qa_2; \\
	\#C(\F_{q^5}) &= q^5+1+a_1^5 -5a_1^3a_2 + 5a_1a_2^2 -5qa_1a_2 -5a_2a_3+5q^2a_1.
\end{align*}
We will complete our proof by showing that the restrictions $a_2 \equiv 0 \pmod 2$ and $a_3 \equiv 1 \pmod 2$ on the coefficients force the point counts to fall into the impossible cases listed in the lemmas above.
First consider the case when $a_1\equiv 0\pmod 2$. Expanding the equations above leads to
  \begin{equation*}
    \begin{aligned}
      \#C(\F_{q})   &\equiv 0 \pmod{2},& 
      \#C(\F_{q^2}) &\equiv 2 \pmod{4},\\  
      \#C(\F_{q^3}) &\equiv 1 \pmod{2},& 
      \#C(\F_{q^4}) &\equiv 2 \pmod{8},\\ 
      \#C(\F_{q^5}) &\equiv 0 \pmod{2}.\\ 
    \end{aligned}
  \end{equation*}
In particular, Lemma~\ref{lem:not1} applies, and we reach a contradiction.

Now assume that $a_1 \equiv 1\pmod 2$. The analogous computation gives
\begin{equation*}
    \begin{aligned}
      \#C(\F_{q})   &\equiv 1 \pmod{2},& 
      \#C(\F_{q^2}) &\equiv 1 \pmod{4},\\  
      \#C(\F_{q^3}) &\equiv 0 \pmod{2},& 
      \#C(\F_{q^4}) &\equiv 5 \pmod{8},\\ 
      \#C(\F_{q^5}) &\equiv 0 \pmod{2}.\\ 
    \end{aligned}
  \end{equation*}
Corollary~\ref{cor:counts} applies, and we again reach a contradiction.
\end{proof}

\begin{remark}
  The argument above can be automated, and by iterating over partitions of $2g + 2$, and applying Proposition~\ref{prop:2adicnew}, one can rule out Weil polynomials modulo $2$, similarly to Proposition~\ref{prop:countingpolys}.
  We have verified that the sets of unrealizable Weil polynomials obtained by these two different procedures agree up to genus $g\leq 10$ by using the Newton identities. 
Indeed, we expect both sets to match for any $g$, although we are unable to prove it in general.
\end{remark}

\section{Experimental data}\label{sec:experimentaldata}

The main results of this paper were inspired by experimental data, which we include here to illustrate the phenomena.
By performing an exhaustive search, the number of Jacobians of hyperelliptic curves in each isogeny class over $\F_q$ (up to isomorphism of principally polarized abelian varieties) has been computed for threefolds and prime powers $q \le 13$ that are either prime or odd.
Similarly, by iterating over the isomorphism classes of smooth plane quartic curves, the number of isogeny classes of abelian threefolds which contain the Jacobian of a smooth plane quartic curve has been computed for $q = 2, 3$, and $5$.
Both searches were done by Andrew Sutherland using the techniques developed in~\cite{smalljac, censusK3}, and the data have been incorporated in the \emph{$L$-functions and Modular Forms Database}~\cite{lmfdb}; see \url{www.LMFDB.org/Variety/Abelian/Fq/}\@.

By combining these data sets, one can also deduce which isogeny classes of abelian threefolds do not contain a Jacobian for $q = 2, 3$, and $5$, see Table~\ref{tab:jacobians}.%
{\renewcommand{\arraystretch}{1.2}
\begin{table}[!ht]
\addtocounter{equation}{1}
\begin{tabular}{l|cc|cc|r}
\multirow{2}{*}{$q$} &
\multicolumn{4}{c|}{Contains Jacobian} &
\multirow{2}{*}{Total}
\\
   & Yes & No & Hyperelliptic & Quartic curve &
\\
\hline
\multirow{2}{*}{\numprint{2}} &
\href{https://www.lmfdb.org/Variety/Abelian/Fq/?g=3&q=2&jacobian=yes}{\numprint{108}} &
\href{https://www.lmfdb.org/Variety/Abelian/Fq/?g=3&q=2&jacobian=no}{\numprint{107}} &
\href{https://www.lmfdb.org/Variety/Abelian/Fq/?q=2&g=3&hyp_cnt=1-}{\numprint{59}} &
\numprint{73} &
\multirow{2}{*}{\numprint{215}}
\\
& 50.23\% & 49.77\% & 27.44\% & 33.95\%
\\
\multirow{2}{*}{\numprint{3}} &
\href{https://www.lmfdb.org/Variety/Abelian/Fq/?g=3&q=3&jacobian=yes}{\numprint{479}} &
\href{https://www.lmfdb.org/Variety/Abelian/Fq/?g=3&q=3&jacobian=no}{\numprint{198}} &
\href{https://www.lmfdb.org/Variety/Abelian/Fq/?q=3&g=3&hyp_cnt=1-}{\numprint{297}} &
\numprint{389} &
\multirow{2}{*}{\numprint{677}}
\\
& 70.75\% & 29.25\% & 43.87\% & 57.46\%
\\
\multirow{2}{*}{\numprint{5}} &
\href{https://www.lmfdb.org/Variety/Abelian/Fq/?g=3&q=5&jacobian=yes}{\numprint{2611}} &
\href{https://www.lmfdb.org/Variety/Abelian/Fq/?g=3&q=5&jacobian=no}{\numprint{342}} &
\href{https://www.lmfdb.org/Variety/Abelian/Fq/?q=5&g=3&hyp_cnt=1-}{\numprint{1723}} &
\numprint{2471} &
\multirow{2}{*}{\numprint{2953}}
\\
& 88.42\% & 11.58\% & 58.35\% & 83.68\%
\end{tabular}

\caption{Types of Jacobians per isogeny classes of abelian threefolds.}
\label{tab:jacobians}
\end{table}
}%
These data sets, and more precisely the multisets of virtual point counts modulo 2 and 4 that can be extracted from this data, provide motivation for Theorem~\ref{main-thm-g3}.

In Figure~\ref{fig:effectiveness-main-thm-g3}, we compare the proportion of isogeny classes of abelian threefolds over $\F_q$ which do not contain a hyperelliptic Jacobian for  $q \leq 13$ with the proportion of such isogeny classes which are ruled out via Theorem~\ref{main-thm-g3}, giving us some insight into the efficiency of Theorem~\ref{main-thm-g3} as $q$ grows.
\begin{figure}[ht!]
  \addtocounter{equation}{1}
	\includegraphics[scale=0.75]{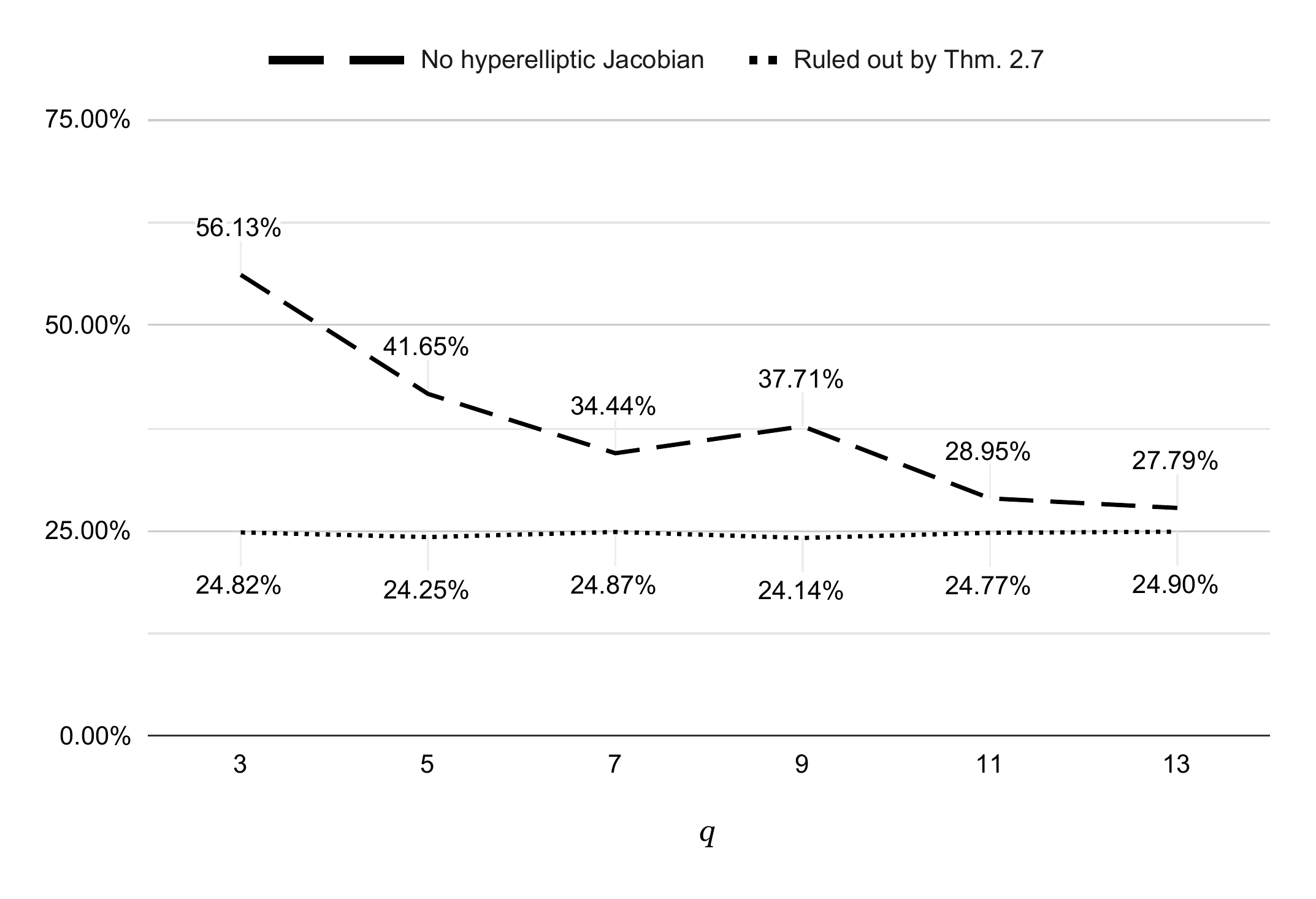}
	\caption{Effectiveness of Theorem~\ref{main-thm-g3}.}
  \label{fig:effectiveness-main-thm-g3}
\end{figure}
Examples of isogeny classes which do not contain a (hyperelliptic) Jacobian but are not ruled out by Theorem~\ref{main-thm-g3} are given by the following Weil polynomials, cf.~isogeny classes \href{https://www.lmfdb.org/Variety/Abelian/Fq/3/3/a_ab_ac}{\texttt{3.3.a\_ab\_ac}} and \href{https://www.lmfdb.org/Variety/Abelian/Fq/3/5/ac_j_aw}{\texttt{3.5.ac\_j\_aw}} on the LMFDB~\cite{lmfdb}:
\begin{align*}
	Z_1(t) &= t^6 -t^4 -2t^3 -3t^2 + 27,\\
  Z_2(t) &= t^6 -2t^5 + 9t^4 -22t^3 + 45t^2 -50t +125.
\end{align*}
In fact, neither isogeny class contains the Jacobian of any curve.  These isogeny classes correspond to the row of Table~\ref{tab:partitions_genus_3} with the most partitions.

In Table~\ref{tab:proportions_genus_3}, we present the number of isogeny classes which are ruled out by Theorem~\ref{main-thm-g3}, for $17 \leq q \leq 31$.
The data show that the proportion of isogeny classes ruled out by Theorem~\ref{main-thm-g3} quickly approaches 25\%, as expected; see Corollary~\ref{cor:hq}.
Furthermore, naive extrapolation seems to indicate that the proportion of isogeny classes of abelian threefolds over $\F_q$ which do not contain a hyperelliptic Jacobian approaches $25\%$ from above as $q\to\infty$ ranges over odd prime powers, while the proportion of the isogeny classes that are ruled out via Theorem~\ref{main-thm-g3} approaches $25\%$ from below.  Similarly, in Tables~\ref{tab:proportions_genus_4} and~\ref{tab:proportions_genus_5} we display the number of isogeny classes which cannot contain a hyperelliptic Jacobian by Proposition~\ref{prop:2torsion}, and as before, we observe proportions which are already very close to the ones attained in the $q$-limit; see Corollary~\ref{cor:hq}.

{\renewcommand{\arraystretch}{1.2}
  \begin{table}[!ht]
    \addtocounter{equation}{1}
    \begin{tabular}{c|ccccccc}
      $q$ & 17 & 19 & 23 & 25 & 27 & 29 & 31
      \\
      \hline
      Number of isogeny classes &
      \numprint{112283} &
      \numprint{156589} &
      \numprint{277517} &
      \numprint{332166} &
      \numprint{333695} &
      \numprint{555843} &
      \numprint{678957}
      \\
      \hline
      \multirow{2}{*}{Inadmissible by Thm.~\ref{main-thm-g3}} &
      \numprint{27974} &
      \numprint{39034} &
      \numprint{69268} &
      \numprint{82564} &
      \numprint{83350} &
      \numprint{138730} &
      \numprint{169574}
      \\
                        &
      24.91\% &
      24.93\% &
      24.96\% &
      24.86\% &
      24.98\% &
      24.96\% &
      24.98\%
    \end{tabular}
    \caption{Isogeny classes ruled out by Proposition~\ref{prop:2torsion} for genus 3 and small $q$.}
    \label{tab:proportions_genus_3}
  \end{table}
}

{\renewcommand{\arraystretch}{1.2}
  \begin{table}[!ht]
    \addtocounter{equation}{1}
    \begin{tabular}{c|cccccc|c}
      $q$ & 3 & 5 & 7 & 9 & 11 & 13 & $q\rightarrow\infty$
      \\
      \hline
      Number of isogeny classes &
      \numprint{10963} &
      \numprint{132839} &
      \numprint{705593} &
      \numprint{2232114} &
      \numprint{6718947} &
      \numprint{15477119}
      \\
      \hline
      \multirow{2}{*}{Inadmissible by Prop.~\ref{prop:2torsion}} &
      \numprint{3856} &
      \numprint{48910} &
      \numprint{262564} &
      \numprint{829189} &
      \numprint{2513570} &
      \numprint{5794772}
      \\
             &
      35.17\% &
      36.82\% &
      37.21\% &
      37.15\% &
      37.41\% &
      37.44\% &
      37.50\%
    \end{tabular}
    \caption{Isogeny classes ruled out by Proposition~\ref{prop:2torsion} for genus 4 and small $q$. }
    \label{tab:proportions_genus_4}
  \end{table}
}

{\renewcommand{\arraystretch}{1.2}
  \begin{table}[!ht]
    \addtocounter{equation}{1}
    \begin{tabular}{c|cc|c}
      $q$ & 3 & 5 & $q\rightarrow\infty$
      \\
      \hline
      Number of isogeny classes &
      \numprint{267465} &
      \numprint{11902325}
      \\
      \hline
      \multirow{2}{*}{Inadmissible by Prop.~\ref{prop:2torsion}} &
      \numprint{137866} &
      \numprint{6286570}
      \\
             &
      51.55\% &
      52.82\% &
      53.12\%
    \end{tabular}
    \caption{Isogeny classes ruled out by Proposition~\ref{prop:2torsion} in genus 5 and small $q$.}
    \label{tab:proportions_genus_5}
  \end{table}
}

To generate Tables~\ref{tab:proportions_genus_3},~\ref{tab:proportions_genus_4} and~\ref{tab:proportions_genus_5}, we enumerated all isogeny classes through their Weil polynomials:
First, we enumerated Weil polynomials of degree $2g$, and then we filtered by the Honda--Tate condition on its factors, see~\cite[Chapter 2]{waterhouse}, to only keep the ones that correspond to an isogeny class of abelian varieties of dimension $g$.
We enumerated Weil polynomials using \defi{root-unitary}\footnote{\url{https://github.com/kedlaya/root-unitary}}, which implements the techniques introduced in~\cite{rootunitary}.

\end{document}